\documentclass{amsart}
\usepackage{setspace}
\usepackage{a4}
\usepackage{amsthm}
\usepackage{latexsym}
\usepackage{amsfonts}
\usepackage{graphicx}
\usepackage{textcomp}
\usepackage{cite}
\usepackage{enumerate}
\usepackage{amssymb}
\usepackage{hyperref}
\usepackage{amsmath}
\usepackage{tikz}
\usepackage[mathscr]{euscript}
\usepackage{mathtools}
\newtheorem{theorem}{Theorem}[section]

\newtheorem{definition}[theorem]{Definition}

\title{This is the title}
\usepackage{fancyhdr}

\begin{document}
\hrule\hrule\hrule\hrule\hrule
\vspace{0.3cm}	
\begin{center}
{\bf{UNEXPECTED  UNCERTAINTY PRINCIPLE FOR DISC BANACH SPACES}}\\
\vspace{0.3cm}
\hrule\hrule\hrule\hrule\hrule
\vspace{0.3cm}
\textbf{K. MAHESH KRISHNA}\\
Post Doctoral Fellow \\
Statistics and Mathematics Unit\\
Indian Statistical Institute, Bangalore Centre\\
Karnataka 560 059, India\\
Email:  kmaheshak@gmail.com\\

Date: \today
\end{center}

\hrule\hrule
\vspace{0.5cm}
\textbf{Abstract}: Let  $(\{f_n\}_{n=1}^\infty, \{\tau_n\}_{n=1}^\infty)$   and $(\{g_n\}_{n=1}^\infty, \{\omega_n\}_{n=1}^\infty)$    be	unbounded continuous p-Schauder frames  ($0<p<1$) for a disc Banach space $\mathcal{X}$. Then for every $x \in ( \mathcal{D}(\theta_f) \cap\mathcal{D}(\theta_g))\setminus\{0\}$,  we show that 
\begin{align}\label{UB}
	\|\theta_f x\|_0\|\theta_g x\|_0 \geq 	\frac{1}{\left(\displaystyle\sup_{n,m \in \mathbb{N} }|f_n(\omega_m)|\right)^p\left(\displaystyle\sup_{n, m \in \mathbb{N}}|g_m(\tau_n)|\right)^p},
\end{align}
where 
\begin{align*}
	&	\theta_f: \mathcal{D}(\theta_f) \ni x \mapsto \theta_fx \coloneqq \{f_n(x)\}_{n=1}^\infty\in \ell^p(\mathbb{N}), \quad 	\theta_g: \mathcal{D}(\theta_g) \ni x \mapsto \theta_gx \coloneqq \{g_n(x)\}_{n=1}^\infty\in \ell^p(\mathbb{N}).	
\end{align*}
Inequality (\ref{UB}) is unexpectedly different from   both  bounded uncertainty principle \textit{[arXiv:2308.00312v1]} and unbounded uncertainty principle \textit{[arXiv:2312.00366v1]} for Banach spaces.



\hrule

\hrule
\section{Introduction}
Given a finite collection $\{\tau_j\}_{j=1}^n$ in a finite dimensional Hilbert space $\mathcal{H}$ over $\mathbb{K}$ ($\mathbb{R}$ or $\mathbb{C}$), define
\begin{align*}
	\theta_\tau: \mathcal{H} \ni h \mapsto \theta_\tau h \coloneqq (\langle h, \tau_j\rangle)_{j=1}^n \in \mathbb{K} ^n.
\end{align*}
Recall that  a collection $\{\tau_j\}_{j=1}^n$ in  $\mathcal{H}$ is said to be a Parseval frame \cite{BENEDETTOFICKUS} for $\mathcal{H}$ if 
\begin{align*}
	\|h\|^2=\sum_{j=1}^{n}|\langle h, \tau_j\rangle|^2, \quad \forall h \in \mathcal{H}.
\end{align*}
Most general form of discrete uncertainty principle for finite dimensional Hilbert spaces is the following.
\begin{theorem} (\textbf{Donoho-Stark-Elad-Bruckstein-Ricaud-Torr\'{e}sani Uncertainty Principle}) \cite{DONOHOSTARK, ELADBRUCKSTEIN, RICAUDTORRESANI} \label{RT}
	Let $\{\tau_j\}_{j=1}^n$,  $\{\omega_j\}_{j=1}^n$ be two Parseval frames   for a  finite dimensional Hilbert space $\mathcal{H}$. Then 
\begin{align*}
\frac{\|\theta_\tau h\|_0^2+\|\theta_\omega h\|_0^2}{2}	\geq \left(\frac{\|\theta_\tau h\|_0+\|\theta_\omega h\|_0}{2}\right)^2	\geq \|\theta_\tau h\|_0\|\theta_\omega h\|_0\geq \frac{1}{\displaystyle\max_{1\leq j, k \leq n}|\langle\tau_j, \omega_k \rangle|^2}, \quad \forall h \in \mathcal{H}\setminus \{0\}.
		\end{align*}	
\end{theorem}
Recently, Theorem \ref{RT} has been derived for Banach spaces using  continuous p-Schauder frames.
\begin{definition}\cite{KRISHNA2}\label{PCSF}
	Let 	$(\Omega, \mu)$ be a measure space. Let    $\{\tau_\alpha\}_{\alpha\in \Omega}$ be a collection in a Banach   space $\mathcal{X}$ and     $\{f_\alpha\}_{\alpha\in \Omega}$ be a collection in  $\mathcal{X}^*$. The pair $(\{f_\alpha\}_{\alpha\in \Omega}, \{\tau_\alpha\}_{\alpha\in \Omega})$   is said to be a \textbf{continuous p-Schauder frame}  for $\mathcal{X}$   ($1\leq p<\infty$) if the following holds. 	
	\begin{enumerate}[\upshape(i)]
		\item For every $x\in \mathcal{X}$, the map $\Omega \ni \alpha \mapsto f_\alpha(x)\in \mathbb{K}$ 	is measurable.
		\item For every $x \in \mathcal{X}$, 
		\begin{align*}
			&\|x\|^p=\int\limits_{\Omega}|f_\alpha(x)|^p\, d \mu(\alpha) \quad \text{if} \quad 1\leq p<\infty, \\
		&\|x\|=\operatorname{ess ~ sup}_{\alpha \in \Omega} |f_\alpha (x)| \quad \text{if} \quad p=\infty.
	\end{align*}
		\item For every $x\in \mathcal{X}$, the map 
		$
		\Omega \ni \alpha \mapsto f_\alpha(x)\tau_\alpha \in \mathcal{X}$
		is weakly measurable.
		\item For every $x \in \mathcal{X}$, 
		\begin{align*}
			x=\int\limits_{\Omega}	f_\alpha (x)\tau_\alpha \, d \mu(\alpha),
		\end{align*}  
		where the 	integral is weak integral.
	\end{enumerate}
\end{definition}
Given a continuous p-Schauder frame $(\{f_\alpha\}_{\alpha\in \Omega}, \{\tau_\alpha\}_{\alpha\in \Omega})$  for $\mathcal{X}$, define 
\begin{align*}
	\theta_f: \mathcal{X} \ni x \mapsto \theta_fx \in \mathcal{L}^p(\Omega, \mu); \quad   \theta_fx: \Omega \ni \alpha \mapsto  (\theta_fx) (\alpha)\coloneqq f_\alpha (x) \in \mathbb{K}
\end{align*}
\begin{theorem}(\textbf{Functional Continuous  Donoho-Stark-Elad-Bruckstein-Ricaud-Torr\'{e}sani Uncertainty Principle}) \cite{KRISHNA2, KRISHNA1, KRISHNA3} \label{MT}
	Let $(\Omega, \mu)$,  $(\Delta, \nu)$ be   measure spaces. Let  $(\{f_\alpha\}_{\alpha\in \Omega}$, $ \{\tau_\alpha\}_{\alpha\in \Omega})$  and   $(\{g_\beta\}_{\beta\in \Delta}$, $ \{\omega_\beta\}_{\beta\in \Delta})$   be	continuous p-Schauder frames  for a Banach space $\mathcal{X}$. Then  
	\begin{enumerate}[\upshape(i)]
		\item for $p>1$, we have 
		\begin{align*}
			&\mu(\operatorname{supp}(\theta_f x))^\frac{1}{p}	\nu(\operatorname{supp}(\theta_g x))^\frac{1}{q} \geq 	\frac{1}{\displaystyle\sup_{\alpha \in \Omega, \beta \in \Delta}|f_\alpha(\omega_\beta)|}, \quad \forall x \in \mathcal{X}\setminus\{0\};\\ 	&\nu(\operatorname{supp}(\theta_g x))^\frac{1}{p}	\mu(\operatorname{supp}(\theta_f x))^\frac{1}{q}\geq \frac{1}{\displaystyle\sup_{\alpha \in \Omega , \beta \in \Delta}|g_\beta(\tau_\alpha)|},\quad \forall x \in \mathcal{X}\setminus\{0\}.
		\end{align*}
	where $q$ is the  conjugate index of $p$.
		\item for $p=1$, we have 
		\begin{align*}
			\mu(\operatorname{supp}(\theta_f x)) \geq 	\frac{1}{\displaystyle\sup_{\alpha \in \Omega, \beta \in \Delta}|f_\alpha(\omega_\beta)|}, \quad  	\nu(\operatorname{supp}(\theta_g x))\geq \frac{1}{\displaystyle\sup_{\alpha \in \Omega , \beta \in \Delta}|g_\beta(\tau_\alpha)|}, \quad \forall x \in \mathcal{X}\setminus\{0\}.
		\end{align*}
		\item for $p=\infty$, we have 
	\begin{align*}
		\nu(\operatorname{supp}(\theta_g x)) \geq 	\frac{1}{\displaystyle\sup_{\alpha \in \Omega, \beta \in \Delta}|f_\alpha(\omega_\beta)|}, \quad  	\mu(\operatorname{supp}(\theta_f x))\geq \frac{1}{\displaystyle\sup_{\alpha \in \Omega , \beta \in \Delta}|g_\beta(\tau_\alpha)|}, \quad \forall x \in \mathcal{X}\setminus\{0\}.
	\end{align*}
	\end{enumerate}
\end{theorem}
An unbounded version of \ref{MT} has been recently derived for unbounded frames. 
 \begin{definition}\label{CPSF} \cite{KRISHNA5}
	Let 	$(\Omega, \mu)$ be a measure space and $1\leq p \leq \infty$. Let    $\{\tau_\alpha\}_{\alpha\in \Omega}$ be a collection in a Banach   space $\mathcal{X}$ and     $\{f_\alpha\}_{\alpha\in \Omega}$ be a collection of linear functionals on $\mathcal{X}$ (which may not be bounded). The pair $(\{f_\alpha\}_{\alpha\in \Omega}, \{\tau_\alpha\}_{\alpha\in \Omega})$   is called an  \textbf{unbounded continuous p-Schauder frame} or \textbf{continuous semi p-Schauder frame} for $\mathcal{X}$    if the following conditions  holds. 	
	\begin{enumerate}[\upshape(i)]
		\item For every $x\in \mathcal{X}$, the map $\Omega \ni \alpha \mapsto f_\alpha(x)\in \mathbb{K}$ 	is measurable.
		\item The map 
		\begin{align*}
			\theta_f: \mathcal{D}(\theta_f) \ni x \mapsto \theta_fx \in \mathcal{L}^p(\Omega, \mu); \quad   \theta_fx: \Omega \ni \alpha \mapsto  (\theta_fx) (\alpha)\coloneqq f_\alpha (x) \in \mathbb{K}
		\end{align*}
		is well-defined (need not be bounded).
		\item For every $x\in \mathcal{X}$, the map 
		$
		\Omega \ni \alpha \mapsto f_\alpha(x)\tau_\alpha \in \mathcal{X}$
		is weakly measurable.
		\item For every $x \in \mathcal{D}(\theta_f)$, 
		\begin{align*}
			x=\int\limits_{\Omega}	f_\alpha (x)\tau_\alpha \, d \mu(\alpha),
		\end{align*}  
		where the 	integral is weak integral.
	\end{enumerate}
\end{definition}
\begin{theorem} (\textbf{Unbounded Donoho-Stark-Elad-Bruckstein-Ricaud-Torr\'{e}sani Uncertainty Principles})\label{UUP} \cite{KRISHNA5}
	Let $(\Omega, \mu)$,  $(\Delta, \nu)$ be   measure spaces and $p=1$ or $p=\infty$. Let  $(\{f_\alpha\}_{\alpha\in \Omega}, \{\tau_\alpha\}_{\alpha\in \Omega})$  and   $(\{g_\beta\}_{\beta\in \Delta}, \{\omega_\beta\}_{\beta\in \Delta})$   be	unbounded continuous p-Schauder frames  for a Banach space $\mathcal{X}$. Then for every $x \in (\mathcal{D}(\theta_f)\cap \mathcal{D}(\theta_g))\setminus\{0\}$,  we have 
	\begin{align*}
		\mu(\operatorname{supp}(\theta_f x))\nu(\operatorname{supp}(\theta_g x)) \geq 	\frac{1}{\left(\displaystyle\sup_{\alpha \in \Omega, \beta \in \Delta}|f_\alpha(\omega_\beta)|\right)\left(\displaystyle\sup_{\alpha \in \Omega , \beta \in \Delta}|g_\beta(\tau_\alpha)|\right)}.
	\end{align*}
\end{theorem}
Above results cover the important Lebesgue spaces for $1\leq p\leq\infty$. The next natural class of spaces is the Lebesgue spaces for $0<p<1$.  Recall that 	for $0<p<1$, we define
\begin{align*}
	\ell^p(\mathbb{N})\coloneqq \{\{a_n\}_{n=1}^\infty: a_n \in \mathbb{K}, \forall n \in \mathbb{N}, \sum_{n=1}^{\infty}|a_n|^p<\infty\}
\end{align*}
equipped with the inhomogeneous norm 
\begin{align*}
	\|\{a_n\}_{n=1}^\infty\|_p\coloneqq \sum_{n=1}^{\infty}|a_n|^p, \quad \forall \{a_n\}_{n=1}^\infty \in  \ell^p(\mathbb{N}).
\end{align*}
In this paper, we derive a surprising result which is counter intuitive  to the feeling we gain from Theorem \ref{MT}. This is why we called the uncertainty principle we obtained   as unexpected uncertainty principle.

\section{Unexpected  Uncertainty Principle}

We start by recalling the  following definition. 
\begin{definition}\cite{KRISHNA4}
Let $\mathcal{X}$  be a  vector   space over $\mathbb{K}$. We say that 	$\mathcal{X}$ is a \textbf{disc Banach space}   if there exists a map called as \textbf{disc norm} $\|\cdot\|:\mathcal{X} \to [0, \infty)$ satisfying the following conditions.
\begin{enumerate}[\upshape(i)]
	\item If $x \in \mathcal{X} $ is such that $\|x\|=0$, then $x=0$.
	\item $\|x+y\|\leq \|x\|+\|y\|$ for all $x, y  \in \mathcal{X}$. 
	\item $\|\lambda x\|\leq |\lambda|\|x\|$ for all $x  \in \mathcal{X}$ and for all $\lambda\in\mathbb{K}$ with $|\lambda|\geq 1$.
		\item $\|\lambda x\|\geq |\lambda|\|x\|$ for all $x  \in \mathcal{X}$ and for all $\lambda\in\mathbb{K}$ with $|\lambda|\leq 1$.
	\item $\mathcal{X}$ is complete w.r.t. the metric $d(x, y)\coloneqq \|x-y\|$ for all $x, y  \in \mathcal{X}$. 
\end{enumerate}
\end{definition}
Banach space frame theory which is modeled on classical Lebesgue sequence spaces \cite{KRISHNAJOHNSON} and the theory of unbounded frames for Hilbert and Banach spaces  \cite{LIULIUZHENG, CHRISTENSEN, ANTOINEBALAZS, ANTOINEBALAZS2, ANTOINETRAPANI, ANTOINECORSOTRAPANI, ANTOINESPECKBACHERTRAPANI} naturally gives the following definition.
 	\begin{definition}
 		Let $\mathcal{X}$ be a disc Banach space.    Let  $\{\tau_n\}_{n=1}^\infty$  be a collection in $\mathcal{X}$ and     $\{f_n\}_{n=1}^\infty$ be a collection of linear functionals on $\mathcal{X}$ (which may not be bounded). The pair $(\{f_n\}_{n=1}^\infty, \{\tau_n\}_{n=1}^\infty)$   is said to be a \textbf{unbounded  p-Schauder frame} ($0<p<1$) or \textbf{semi p-Schauder frame} for $\mathcal{X}$    if the following conditions  holds. 	
 		\begin{enumerate}[\upshape(i)]
	\item The map 
 			\begin{align*}
 				\theta_f: \mathcal{D}(\theta_f) \ni x \mapsto \theta_fx \coloneqq \{f_n(x)\}_{n=1}^\infty\in \ell^p(\mathbb{N})
 			\end{align*}
 			is well-defined (need not be bounded).
 			\item For every $x \in \mathcal{D}(\theta_f)$, 
 			\begin{align*}
 				x=\sum_{n=1}^{\infty}	f_n (x)\tau_n
 			\end{align*}  
 			\end{enumerate}
 	\end{definition}
 We are going to use the following important result. 
 		\begin{theorem}\cite{GARLING, HERMAN}\label{GARLING}
 		For every $0<p<1$, 
 		\begin{align*}
 			\left( \sum_{n=1}^{\infty}|a_n|\right)^p\leq  \sum_{n=1}^{\infty}|a_n|^p, \quad \forall \{a_n\}_{n=1}^\infty \in  \ell^p(\mathbb{N}).	
 		\end{align*}	
 	\end{theorem}
 Following is the main result of the paper.
 	\begin{theorem}\label{DISCUP}
 	 Let  $(\{f_n\}_{n=1}^\infty, \{\tau_n\}_{n=1}^\infty)$   and $(\{g_n\}_{n=1}^\infty, \{\omega_n\}_{n=1}^\infty)$     be	unbounded p-Schauder frames  for a disc Banach space $\mathcal{X}$. Then for every $x \in (\mathcal{D}(\theta_f)\cap \mathcal{D}(\theta_g))\setminus\{0\}$,  we have 
 	\begin{align*}
 		\|\theta_f x\|_0\|\theta_g x\|_0 \geq 	\frac{1}{\left(\displaystyle\sup_{n,m \in \mathbb{N} }|f_n(\omega_m)|\right)^p\left(\displaystyle\sup_{n, m \in \mathbb{N}}|g_m(\tau_n)|\right)^p}.
 \end{align*}
 	\end{theorem}
 	\begin{proof}
 		Let $x \in \mathcal{D}(\theta_f)\setminus\{0\}$. Then using Theorem \ref{GARLING}, 	
 		\begin{align*}
 \|\theta_fx\|&= \sum_{n=1}^{\infty}|f_n(x)|^p=\sum_{n \in \operatorname{supp}(\theta_fx)}|f_n(x)|^p=\sum_{n \in \operatorname{supp}(\theta_fx)}\left|f_n\left(\sum_{m=1}^{\infty}	g_m (x)\omega_m\right)\right|^p\\
 &=\sum_{n \in \operatorname{supp}(\theta_fx)}\left|\sum_{m \in \operatorname{supp}(\theta_gx)}g_m (x)f_n(\omega_m)\right|^p\leq \sum_{n \in \operatorname{supp}(\theta_fx)}\left(\sum_{m \in  \operatorname{supp}(\theta_gx)}|g_m (x)f_n(\omega_m)|\right)^p\\
 &\leq \left(\displaystyle\sup_{n,m \in \mathbb{N} }|f_n(\omega_m)|\right)^p\sum_{n \in \operatorname{supp}(\theta_fx)}\left(\sum_{m \in \operatorname{supp}(\theta_gx)}|g_m (x)|\right)^p\\
 			&=\left(\displaystyle\sup_{n,m \in \mathbb{N} }|f_n(\omega_m)|\right)^p\|\theta_fx\|_0\left(\sum_{m=1}^{\infty}|g_m (x)|\right)^p\leq \left(\displaystyle\sup_{n,m \in \mathbb{N} }|f_n(\omega_m)|\right)^p\|\theta_fx\|_0\left(\sum_{m=1}^{\infty}|g_m (x)|^p\right)\\
 			&=\left(\displaystyle\sup_{n,m \in \mathbb{N} }|f_n(\omega_m)|\right)^p\|\theta_fx\|_0\|\theta_gx\|.
 		\end{align*}
 		Therefore 
 		\begin{align}\label{FI}
 			\frac{1}{\left(\displaystyle\sup_{n,m \in \mathbb{N} }|f_n(\omega_m)|\right)^p}\|\theta_fx\|\leq \|\theta_fx\|_0	\|\theta_gx\|.
 		\end{align}
 	On the other hand, let  $x \in \mathcal{D}(\theta_g)\setminus\{0\}$. Then again using Theorem \ref{GARLING}, 	
 	\begin{align*}
 		\|\theta_gx\|&= \sum_{m=1}^{\infty}|g_m(x)|^p=\sum_{m \in \operatorname{supp}(\theta_gx)}|g_m(x)|^p=\sum_{m \in \operatorname{supp}(\theta_fx)}\left|g_m\left(\sum_{n=1}^{\infty}	f_n(x)\tau_n\right)\right|^p\\
 		&=\sum_{m \in \operatorname{supp}(\theta_gx)}\left|\sum_{n \in \operatorname{supp}(\theta_fx)}f_n(x)g_m(\tau_n)\right|^p\leq \sum_{m \in  \operatorname{supp}(\theta_gx)}\left(\sum_{n \in  \operatorname{supp}(\theta_fx)}|f_n(x) (x)g_m(\tau_n)|\right)^p\\
 		&\leq \left(\displaystyle\sup_{n,m \in \mathbb{N} }|g_m(\tau_n)|\right)^p\sum_{m \in \operatorname{supp}(\theta_gx)}\left(\sum_{n \in \operatorname{supp}(\theta_fx)}|f_n (x)|\right)^p\\
 		&=\left(\displaystyle\sup_{n,m \in \mathbb{N} }|g_m(\tau_n)|\right)^p\|\theta_gx\|_0\left(\sum_{n=1}^{\infty}|f_n (x)|\right)^p\leq \left(\displaystyle\sup_{n,m \in \mathbb{N} }|g_m(\tau_n)|\right)^p\|\theta_gx\|_0\left(\sum_{n=1}^{\infty}|f_n (x)|^p\right)\\
 		&=\left(\displaystyle\sup_{n,m \in \mathbb{N} }|g_m(\tau_n)|\right)^p\|\theta_gx\|_0\|\theta_fx\|.
 	\end{align*}	
 Therefore 
 \begin{align}\label{SI}
\frac{1}{\left(\displaystyle\sup_{n,m \in \mathbb{N} }|g_m(\tau_n)|\right)^p}\|\theta_gx\|\leq \|\theta_gx\|_0	\|\theta_fx\|. 	
 \end{align}
 	Multiplying Inequalities (\ref{FI}) and (\ref{SI}) we get 
 	\begin{align*}
 \frac{1}{\left(\displaystyle\sup_{n,m \in \mathbb{N} }|f_n(\omega_m)|\right)^p\left(\displaystyle\sup_{n, m \in \mathbb{N}}|g_m(\tau_n)|\right)^p} \|\theta_fx\|	\|\theta_gx\|\leq &\|\theta_fx\|_0\|\theta_gx\|_0\|\theta_fx\|\|\theta_gx\|, \\
 		& \quad \forall x \in (\mathcal{D}(\theta_f)\cap \mathcal{D}(\theta_g))\setminus\{0\}.
 	\end{align*}
 	A cancellation of $\|\theta_fx\|\|\theta_gx\|$ gives the  inequality.	
 		\end{proof}
 		As continuous version of  Theorem \ref{GARLING}  fails (even for finite measure spaces) it seems  that continuous version of  Theorem \ref{DISCUP} fails.\\
 		In view of Tao's uncertainty principle \cite{TAO} we believe that Theorem \ref{GARLING} can be improved in prime dimensions.

 \bibliographystyle{plain}
 \bibliography{reference.bib}

\end{document}